\documentclass[psamsfonts,fceqn,leqno]{amsart}
  \usepackage{amsthm}
  \usepackage{amsmath}
  \usepackage{amssymb}
  \usepackage{mathrsfs}
  
  \newtheorem{theorem}{Theorem}[section]
  \newtheorem{corollary}[theorem]{Corollary}
  \newtheorem{proposition}[theorem]{Proposition}
  \newtheorem{lemma}[theorem]{Lemma}
  \newtheorem{question}[theorem]{Question}
  
  \theoremstyle{definition}

  \newtheorem{example}[theorem]{Example}

  \def\R{{\mathbb R}}
  
  \def\yy{{\mathcal Y}}
  \def\zz{{\mathcal Z}}
  \def\fii{\varphi}

  \numberwithin{equation}{section}

  \title[Wijsman Hyperspaces: Subspaces and Embeddings]
  {Wijsman Hyperspaces: Subspaces and\\ Embeddings}
  \author[J. Cao]{Jiling Cao$^\dagger$}
  \address{School of Computing and Mathematical Sciences,
  Auckland University of Technology, Private Bag 92006, Auckland
  1142, New Zealand}
  \email{jiling.cao@aut.ac.nz}
  \author[H.~J.~K. Junnila]{Heikki J. K. Junnila$^\ddagger$}
  \address{Department of Mathematics and Statistics, The
  University of Helsinki, P.~O. Box 68, FI-00014,  Helsinki,
  Finland}
  \email{heikki.junnila@helsinki.fi}
  \author[W.~B. Moors]{Warren B. Moors}
  \address{Department of Mathematics, The University of
  Auckland, Private Bag 92019, Auckland, New Zealand}
  \email{moors@math.auckland.ac.nz}
  \thanks{\hspace{-1.66em} 2010 \emph{Mathematics Subject
  Classification.}
  Primary 54B20, 54E35; Secondary 54C25, 54B10.}
  \thanks{\noindent $^\dagger$This paper was partially
  written when the first author was in a Research and Study
  Leave from July to December 2009. He would like to thank the
  Department of Mathematics and Statistics at the University
  of Helsinki for hospitality.}
  \thanks{\noindent $^\ddagger$The paper was eventually completed
  when the second author visited Auckland in November 2010.
  He would like to thank the School of Computing and
  Mathematical Sciences at the Auckland University of
  Technology for hospitality.}

  \date{}

  \dedicatory{Dedicated to the memory of Professor Jun-iti
  Nagata}

\begin{document}

  \begin{abstract}
  In this paper, topological properties of Wijsman hyperspaces
  are investigated. We study the existence of isolated points in
  Wijsman hyperspaces. We show
  that every Tychonoff space can be embedded as a closed subspace
  in the Wijsman hyperspace of a complete metric space which is
  locally $\mathbb R$.
  \end{abstract}

  \maketitle

  \section{Introduction}

  In this paper, we consider the set $2^{X}$ consisting of all
  non-empty closed subsets of a metric space $(X,d)$, equipped
  with the \emph{Wijsman topology} $\tau_{w(d)}$. In the following,
  we denote the \emph{Wijsman hyperspace} $(2^{X},\tau_{w(d)})$
  of a metric space $(X,d)$ by $2^{(X,d)}$. The easiest way to
  describe the space $2^{(X,d)}$ is through an identification:
  identify a non-empty closed subset $S$ of $X$ with the distance
  function $d(\cdot,S)$. In this way $2^{X}$ is identified with
  a subset of the function space $C(X)$ and the Wijsman topology
  is the topology of pointwise convergence on this subset.

  We refer the reader to \cite{beer:93} and \cite{beer:94} for
  background on Wijsman hyperspaces.

  Lechicki and Levi showed in \cite{lechicki-levi:87} that the
  Wijsman hyperspace of a separable metric space is metrizable.
  There has been a considerable effort to explore completeness
  properties of Wijsman topologies, and one line of research was
  completed with the result of Costantini \cite{costantini:95}
  that the Wijsman hyperspace of a Polish metric space is Polish.

  Besides metrizability and various completeness properties,
  other topological properties of Wijsman hyperspaces have not
  been widely studied. In this paper we give results which show
  that Wijsman hyperspaces of topologically simple non-separable
  metric spaces can have very complicated topologies. The first
  result of this kind in the literature is an example, due to
  Costantini \cite{costantini:98} of an uncountable discrete
  complete metric space $(X,d)$ such that $2^{(X,d)}$ is not
  \v Cech-complete. In Section 3 below, we give a general result
  which yields Costantini's result as well as some later results
  as corollaries.

  \section{Isolated points of Wijsman hyperspaces}
  \label{sec:overview}

  In this section, we give three examples to demonstrate
  various possibilities on the existence of isolated
  points in Wijsman hyperspaces.

  The first example is due to Chaber and Pol
  \cite[Remark 3.1]{chaber-pol:02}.

  \begin{example} \label{exam:discrete2}
  \emph{Let $\delta$ be the $0$-$1$ metric on a set $X$. Then
  $2^{(X,\delta)}$ is homeomorphic to
  $\{0,1\}^X\setminus\{ {\bf 0}\}$, where $\{ 0,1\}$ is
  discrete and ${\bf 0}$ is the constant function
  with value 0. If $X$ is infinite, then
  $2^X$ has no isolated points.}
  \end{example}

  The above example and well-known properties of the Cantor cube
  $\{0,1\}^{X}$ show that the Wijsman hyperspace
  $2^{(X,\delta)}$ of a 0-1 metric space $(X,\delta)$
  is locally compact and satisfies the countable chain condition.
  Moreover, if $|X|\le 2^{\omega}$, then $2^{(X,\delta)}$ is separable.

  The Wijsman hyperspace of an infinite metric
  space is non-discrete. Nevertheless, Wijsman hyperspaces
  may have many isolated points.

  \begin{example} \label{exam:discrete3}
  \emph{For every set $X$, there exists a discrete metric $d$
  on $X$ such that every singleton subset of $X$ is an isolated point
  in $2^{(X,d)}$.}
  \end{example}

  \begin{proof}
  To avoid a triviality, let $X$ be infinite. Express $X$ as the
  union of a family of pairwise disjoint two-point subsets, that is,
  write $X=\bigcup \{X_\alpha: \alpha\in A\}$, where $X_\alpha=
  \{x_\alpha^0, x_\alpha^1\}$ with $x_\alpha^0\ne x_\alpha^1$
  for each $\alpha \in A$ and
  $X_\alpha\cap X_\beta =\emptyset$ whenever $\alpha \ne\beta$.
  Define $d: X \times X \to \{0,1,2\}$ by setting
  \[
  d(x,y) =\left \{
  \begin{array}{ll}
  0, & \mbox{if $x= y\,$;}\\
  2, & \mbox{if $\{x,y\}=X_\alpha$ for some $\alpha\,$;}\\
  1, & \mbox{otherwise.}
  \end{array}
  \right.
  \]
  It can be checked that $d$ is a metric on $X$. Let $\alpha\in A$
  and $i \in \{0,1\}$. Note that $d(x_\alpha^{1-i},x_\alpha^{i})=2$
  and $d(x_\alpha^{1-i},z)\leq 1$ for every $z\ne x_\alpha^i$.
  As a consequence, we have
  $\left\{ F \in 2^X: d(x_\alpha^{1-i},F)>1\right\}
  =\{\{x_\alpha^i\}\}$, and the set
  $\{\{x_\alpha^i\}\}$ is thus open in $2^{(X,d)}$.
  \end{proof}

  \begin{example} \label{exam:discrete4}
  \emph{A discrete metric $d$ on $\mathbb N$ such that every
  non-empty finite subset of $\mathbb N$ is an
  isolated point in $2^{({\mathbb N},d)}$.}
  \end{example}

  \begin{proof}
  Define $d$ by the formula $d(n,k)=|2^{-n} -2^{-k}|$ for all
  $n, k \in \mathbb N$, and note that $d$ is a discrete metric
  on $\mathbb N$.

  Let $E \subset \mathbb N$ be non-empty and finite.
  Set $m=2+\max E$. The set
  \[
  {\mathcal W} =\{F \in 2^{\mathbb N}: |d(k, F)-d(k,E)|
  <2^{-m-1} \mbox{ for each } k \le m\}
  \]
  is a neighborhood of $E$ in $2^{({\mathbb N},d)}$. We show that
  ${\mathcal W} =\{E\}$.

  Let $F \in {\mathcal W}$. To show that $F=E$, we first show
  that $F \cap [1,m] =E \cap [1,m]$. Let $k\le m$. Since
  $F\in{\mathcal W}$,
  we have $|d(k, F)-d(k,E)|<2^{-m-1}$. Thus, if $k\in E$, then
  $d(k,F) <2^{-m-1}$. Since
  \[
  d(k, \mathbb N \setminus \{k\})= d(k, k+1) =2^{-k} -
  2^{-k-1} = 2^{-k-1}\geq 2^{-m-1},
  \]
  it follows that $k\in F$. On the other hand, if $k\not
  \in E$, then
  \[
  d(k,E)\ge d(k, \mathbb N \setminus \{k\})
  = 2^{-k-1}\geq 2^{-m-1}
  \]
  and it follows that
  \[
  d(k,F) > d(k,E)-2^{-m-1} \ge 2^{-m-1} -2^{-m-1} =0,
  \]
  which implies that $k\not \in F$. Hence, we have shown that
  $F\cap [1,m] =E\cap [1, m].$

  To conclude the proof of $F=E$, it suffices to show that
  $F\subseteq [1,m]$. Assume on the contrary that there exists
  $k\in F$ with $k>m$. Then we have that
  \[d(m,F)\leq d(m,k)=2^{-m}-2^{-k}<2^{-m}.
  \]
  Since $m=2+\max E$, we have that
  \[d(m,E)=2^{-m+2}-2^{-m}>2^{-m+1},
  \]
  and it follows that
  \[|d(m,F)-d(m,E)|>2^{-m+1}-2^{-m}=2^{-m}.
  \]
  This, however, is a contradiction since $F\in{\mathcal W}$.
  \end{proof}

  \begin{question}
  Does there exist an uncountable metric space $(X,d)$ such that
  every non-empty finite subset of $X$ is an isolated point
  in $2^{(X,d)}$?
  \end{question}

  \section{ On subspaces of
  Wijsman Hyperspaces}\label{sec:embeddings}

  In this section, we investigate those topological spaces
  which can be embedded as closed subspaces in Wijsman
  hyperspaces of various metric spaces. We start with a
  simple observation.

  \begin{proposition} \label{prop:01metric}
  A $T_1$-space $T$ is zero-dimensional and locally
  compact if, and only if, $T$ embeds as a closed subspace
  in $2^{(X,\delta)}$ for some 0-1 metric space $(X,\delta)$.
  \end{proposition}

  \begin{proof} Sufficiency follows immediately from Example
  \ref{exam:discrete2}.
  To prove necessity, let $T$ be a zero-dimensional and locally
  compact $T_1$-space. The one-point compactification $T^*=T \cup
  \{\infty\}$ is zero-dimensional and compact. There exists an
  embedding $\fii: T^*\to\{0,1\}^\kappa$ for some cardinal
  $\kappa$. The compact set $\fii(T^*)$ is closed in
  $\{0,1\}^\kappa$, and it follows that $\fii|T$ is an embedding
  of $T$ onto a closed subspace of
  $\{0,1\}^\kappa \setminus \{\fii(\infty)\}$. By homogeneity of
  $\{0,1 \}^\kappa$ and Example \ref{exam:discrete2} again,
  $\{0,1\}^\kappa \setminus \{\fii(\infty)\}$ is homeomorphic to
  the Wijsman hyperspace $2^{(X,\delta)}$
  of a 0-1 metric space $(X, \delta)$.
  \end{proof}

  To obtain some deeper results on embeddings, we first
  establish the following key lemma.

  \begin{lemma}\label{lem:multi}
  Let $\{(X_\alpha, d_\alpha): \alpha \in A\}$ be a family
  of mutually disjoint complete metric spaces such that for
  each $\alpha\in I$, the set $E_\alpha =d_\alpha(X_\alpha
  \times X_\alpha)$ is a subset
  of the closed unit interval $\mathbb I$. Then there
  exists a compatible complete metric $d$ on the free sum
  $X=\bigoplus_{\alpha\in A}X_\alpha$ such that the
  product space $\prod_{\alpha\in A}2^{(X_\alpha,d_\alpha)}$
  embeds as a closed subspace in $2^{(X,d)}$.
  Moreover, $d(X\times X)\subseteq \bigcup_{\alpha\in A}
  E_\alpha\cup\{2\}$.
  \end{lemma}

  \begin{proof}
  Set
  ${\mathcal Y}=\prod_{\alpha\in A}2^{(X_\alpha,d_\alpha)}$.
  Equip $X$ with the metric $d$ defined by
  \[
  d(x,y) =\left \{
  \begin{array}{ll}
  d_\alpha(x,y),
  & \mbox{if $x, y\in
  X_\alpha$ for some $\alpha \in A$;}\\[0.5em]
  2, & \mbox{otherwise.}
  \end{array}
  \right.
  \]
  It is easy to see that $d$ is a compatible complete metric
  for $X$ and the inclusion $d(X\times X)\subseteq
  \bigcup_{\alpha\in A} E_\alpha\cup\{2\}$ holds.
  To complete the proof, we show that ${\mathcal
  Y}$ embeds as a closed subspace in $2^{(X,d)}$. To this
  end, we define a mapping $\varphi: {\mathcal Y} \to
  2^{(X,d)}$ by the formula
  \[
  \varphi(\langle F_\alpha\rangle_{\alpha \in A})= \bigcup
  \{F_\alpha: \alpha \in A\}\,.
  \]

  Denote by $\zz$ the subspace $\fii(\yy)$ of $2^{(X,d)}$.
  Note that we have
  $\zz=\{F\in 2^{X}: F\cap X_\alpha\ne\emptyset\mbox{ for every
  }\alpha\in A\}$. The mapping $\fii$ has an inverse $\fii^{-1}:
  \zz\to\yy$ defined by the formula
  $\fii^{-1}(F)=\langle F\cap X_\alpha\rangle_{\alpha \in A}$. As a
  consequence, $\varphi$ is one-to-one. Next we verify that $\varphi$
  is a homeomorphism $\yy\to\zz$.

  The space $\yy$ has a subbase consisting of sets of the form
  \[
  \Gamma_{\nu,x,a,b}=\{\langle F_\alpha\rangle_{\alpha \in A}\in\yy:
  a<d_\nu(x,F_\nu)<b\}\,,
  \]
  where $\nu\in A$, $x\in X_\nu$
  and $a,b\in\R$.

  Note that if $x\in X_\nu$ and $F\in\zz$, then
  $d(x,F)=d(x,F\cap X_\nu)=d_\nu(x,F\cap X_\nu)$. It follows that
  we have
  \[
  \fii(\Gamma_{\nu,x,a,b})=\{F\in\zz: a<d_\nu(x,F\cap X_\nu)<b\}
  =\{F\in\zz: a<d(x,F)<b\}\,.
  \]
  The relative Wijsman topology of $\zz$ has a subbase consisting
  of sets of the form $\{F\in\zz: a<d(x,F)<b\}$, where
  $x\in X$ and $a,b\in\R$. Hence we have shown that the one-to-one
  mapping $\fii$ transforms a subbase of $\yy$ onto a subbase of
  $\zz$. As a consequence, $\fii$ is a homeomorphism $\yy\to\zz$.

  Finally, we verify that $\zz$ is a closed
  subspace of $2^{(X,d)}$. Let $F\in 2^X\setminus\zz$.
  Then $F \cap X_\alpha =\emptyset$ for some $\alpha\in A$.
  Pick a point $x_0 \in X_\alpha$, and put
  \[
  \mathcal U =\left\{B\in 2^X: d(x_0,B) > 1
  \right\}.
  \]
  Then $\mathcal U$ is an open neighborhood
  of $F$ in $2^{(X,d)}$ such that
  ${\mathcal U} \cap\zz = \emptyset$.
  \end{proof}

  \begin{lemma} \label{lem:finite}
  If $d$ is a finite-valued metric on a nonempty set $X$,
  then $2^{(X,d)}$ is zero-dimensional.
  \end{lemma}

  \begin{proof}
  Let $d(X\times X) =E$. Then $2^{(X,d)}$
  embeds in $C_p(X,E) \subseteq E^X$, where $E$ is equipped
  with the discrete topology. Thus, the conclusion follows.
  \end{proof}

  A metric space $(X,d)$ is \emph{uniformly discrete} if
  $X$ is $\varepsilon$-discrete for some $\varepsilon>0$.

  \begin{question} \label{ques:discrete}
  Can we replace ``finite-valued" in Lemma \ref{lem:finite}
  by ``discrete" or ``uniformly discrete"?
  \end{question}

  We only have the following partial answer to Question
  \ref{ques:discrete}. Following \cite{en:89}, we call a
  topological space $X$
  \emph{totally disconnected} if every quasi-component of
  $X$ is a singleton.

  \begin{proposition}
  Let $(X,d)$ be a discrete metric space. Then $2^{(X,d)}$
  is totally disconnected.
  \end{proposition}

  \begin{proof}
  Let $x\in X$. Let $r_x={{1\over 2}}d(x,X\setminus\{x\})$ and
  note that $r_x>0$. The set
  \[
  {\mathcal G}_x = \{F\in 2^X: d(x,F)<r_x\}=\{F\in 2^{X}: x\in F\}
  \]
  is open in $2^{(X,d)}$, and it is also closed, because
  $2^{X}\setminus G_x=\{F\in 2^X: d(x,F)>r_x\}$.

  The family $\{{\mathcal G}_x: x\in X\}$ of clopen sets
  separates the points of $2^X$ and hence $2^{(X,d)}$ is totally
  disconnected.
  \end{proof}

  \begin{theorem} \label{thm:0-dim}
  A $T_1$-space $T$ is zero-dimensional if, and only if,
  $T$ embeds as a closed subspace in the Wijsman hyperspace
  $2^{(X,d)}$ of a metric space $(X,d)$
  with a 3-valued metric $d$ on $X$.
  \end{theorem}

  \begin{proof}
  Sufficiency follows immediately from Lemma \ref{lem:finite}.
  To prove necessity, suppose that $T$ is zero-dimensional.
  Take a base ${\mathfrak B} =\{B_\alpha: \alpha<\kappa\}$
  consisting of clopen subsets. It is well-known that $T$
  embeds in $\{0,1\}^\kappa$ by the mapping $\varphi: T \to
  \{0,1\}^\kappa$ defined by $\varphi(x)=\chi_{A_x}$, where
  $\chi_{A_x}$ is the characteristic function of the set
  $A_x=\{\alpha <\kappa: x \in B_\alpha\}$. So, we can
  assume that $T\subseteq \{0,1\}^\kappa$. Let $\widetilde T = \{0,1
  \}^\kappa \setminus T$. For every $y\in \widetilde T$, let $Y_y =
  \{0,1\}^\kappa \setminus \{y\}$. By homogeneity of $\{0,
  1\}^\kappa$ and Example \ref{exam:discrete2}, for each
  $y \in \widetilde T$, there exists a 0-1 metric space $(X_y, d_y)$
  such that $Y_y$ is homeomorphic to
  $2^{(X_y,d_y)}$. We can choose the sets $X_y$,
  $y\in\widetilde T$, to be
  mutually disjoint. It follows from
  Lemma \ref{lem:multi} that there exists a 3-valued metric
  $d$ (with values 0, 1, 2) on the free
  sum $X=\bigoplus\{X_y: y \in \widetilde T\}$ such that the product
  space $\prod\{Y_y: y \in \widetilde T\}$ embeds as a closed subspace
  in $2^{(X,d)}$. Finally, it is
  routine to check that the diagonal
  \[
  \Delta = \left\{\langle a_y \rangle_{y\in \widetilde T} \in
  \prod\left\{Y_y: y \in \widetilde T\right\}: a_y =
  a_{y'} \mbox{ for all } y, y' \in \widetilde T\right\}
  \]
  is a closed subspace of $\prod\{Y_y: y \in \widetilde T\}$ which
  is homeomorphic to $T$.
  Therefore, we conclude that $T$ embeds as a closed
  subspace in $2^{(X,d)}$.
  \end{proof}

  Note that Proposition \ref{prop:01metric} and Theorem
  \ref{thm:0-dim} explain why Costantini was able to use
  a 3-valued but not 2-valued metric in his example
  of a complete metric space whose Wijsman
  hyperspace is not \v{C}ech-complete. Let us also note
  that as a consequence of Theorem \ref{thm:0-dim}, we can
  give a ``3-valued solution" to Zsilinszky's problem
  in \cite{zsilinszky:98}. The original solution, by
  Chaber and Pol \cite{chaber-pol:02}, used a non-discrete
  metric space.

  \begin{corollary} \label{coro:hbaire}
  The space $\mathbb Q$ of rationals embeds as a closed
  subspace in $2^{(X,d)}$ for some 3-valued metric space $(X,d)$.
  Consequently, there exists a 3-valued metric space
  whose Wijsman hyperspace is not hereditarily Baire.
  \end{corollary}

  We close the paper with an embedding result for
  Tychonoff spaces. It provides a generic solution to
  problems dealing with closed-hereditary properties of
  Wijsman hyperspaces, and hence it extends some earlier results
  such as those by Costantini, Chaber and Pol mentioned above.

  \begin{theorem} \label{thm:universal}
  Every Tychonoff space can be embedded as a closed
  subspace in the Wijsman hyperspace of a complete metric
  space which is locally $\mathbb R$.
  \end{theorem}

  \begin{proof}
  Let $T$ be a Tychonoff space.
  By a classical result, there exists an infinite cardinal $\kappa$ such that
  $T$ can be embedded into the Tychonoff cube ${\mathbb I}^\kappa$.
  Hence, we may assume that $T \subseteq {\mathbb I}^\kappa$.

  We show that for every $y \in {\mathbb I}^\kappa$, there
  is a locally $\mathbb R$ and complete metric space $(Z_y,
  d_y)$ such that ${\mathbb I}^\kappa \setminus \{y\}$ embeds
  as a closed set in $2^{(Z_y,d_y)}$.
  Since ${\mathbb I}^\kappa$ is homogeneous (see
  \cite{keller:31} and \cite{mill:07}), it suffices to show
  that the assertion holds for $y={\bf 0}$. Let $I$ be the open
  interval $(0, 2)$ in $\mathbb R$. Since $I$ is completely
  metrizable, it admits a compatible complete metric
  $\rho$ which is bounded by 1. Consider the set 
  $
  Z_{{\bf 0}} =\bigcup_{\alpha<\kappa}\left(I \times \{\alpha\}\right),
  $
  and define a metric $d_{{\bf 0}}$ of $Z_{{\bf 0}}$ by the formula
  \[
  d_{{\bf 0}}(\langle x,\alpha\rangle, \langle y,\beta\rangle)=
  \left \{\begin{array}{ll}
  \rho(x,y),
  & \mbox{if $\alpha =\beta$;}\\[0.5em]
  2, & \mbox{otherwise.}
  \end{array}
  \right.
  \]
  Then $(Z_{{\bf 0}}, d_{{\bf 0}})$ is a complete metric space,
  which is locally $\mathbb R$.

  We show that ${\mathbb I}^\kappa \setminus
  \{{\bf 0}\}$ embeds as a closed subspace in
  $2^{(Z_{\bf 0},d_{\bf 0})}$. Consider the mapping
  $\varphi: {\mathbb I}^\kappa\setminus\{{\bf 0}\}
  \to 2^{Z_{\bf 0}}$ defined by the formula 
  \[
  \varphi(y) = \bigcup\left\{(0, y_\alpha]\times \{
  \alpha\}: \alpha <\kappa \right\}.
  \]
  It is easy to see that $\varphi$ is one-to-one
  and the set $\varphi({\mathbb I}^\kappa \setminus
  \{{\bf 0}\})$ is closed in $2^{(Z_{\bf 0},d_{\bf 0})}$.
  Like in the proof of Lemma \ref{lem:multi}, we can show here
  that $\varphi$ transforms a subbase of
  ${\mathbb I}^\kappa\setminus\{{\bf 0}\}$ onto a subbase of
  the subspace $\varphi({\mathbb I}^\kappa \setminus
  \{{\bf 0}\})$ of $2^{(Z_{\bf 0},d_{\bf 0})}$. As a consequence,
  $\varphi$ is an embedding.

  Let $\widetilde T = {\mathbb I}^\kappa\setminus T$. By the foregoing, we
  know that for every $y \in \widetilde T$, there exists a locally $\mathbb R$
  and complete metric space $(Z_y, d_y)$ such that
  $Y_y={\mathbb I}^\kappa \setminus \{y \}$ embeds as a
  closed subspace in $2^{(Z_y,d_y)}$.
  Applying Lemma \ref{lem:multi}, we see that there
  exists a locally $\mathbb R$ and complete metric space $(X,d)$
  such that the product space $\prod \{Y_y: y\in \widetilde T\}$ embeds as
  a closed subspace in $2^{(X,d)}$.
  As in the proof of Theorem \ref{thm:0-dim},
  we see that $T$ is homeomorphic with the closed subspace
  $\Delta$ of $\prod \{Y_y: y \in \widetilde T\}$.
  This completes the proof.
  \end{proof}


\begin{thebibliography}{10}

%
%
  \bibitem{beer:93}
  G. Beer, \emph{Topologies on closed and closed convex sets},
  Kluwer, Dordrecht, 1993.

  \bibitem{beer:94}
  G. Beer, \emph{Wijsman convergence: a survey}, Set-Valued Anal.
  {\bf 2} (1994), 77--94.

%
%
%
  \bibitem{chaber-pol:02}
  J. Chaber and R. Pol, \emph{Note on the Wijsman hyperspaces
  of completely metrizable spaces}, Boll. U. M. I. {\bf 85-B}
  (2002), 827--832.
%
  \bibitem{costantini:95}
  C. Costantini, \emph{Every Wijsman topology relative to a
  Polish space is Polish}, Proc. Amer. Math. Soc. {\bf 123}
  (1995), 2569--2574.

  \bibitem{costantini:98}
  C. Costantini, \emph{On the hyperspace of a non-separable
  metric space}, Proc. Amer. Math. Soc. {\bf 126} (1998),
  3393--3396.
%
%
%
%
  \bibitem{en:89} {R. Engelking}, {\em General topology\/},
  Revised and completed edition, Heldermann Verlag, Berlin,
  1989.
%
%
%
%
%
%
  \bibitem{keller:31} O. H. Keller, \emph{Die homoiomorphie
  der kompakten konvexen Mengen in Hilbertschen rum}, Math.
  Ann. {\bf 105} (1931), 748--758.
%
  \bibitem{lechicki-levi:87}
  A. Lechicki and S. Levi, \emph{Wijsman convergence in the
  hyperspace of a metric space}, Boll. Un. Mat. Ital. (7)
  {\bf 1-B} (1987), 439--452.
%
%
%
  \bibitem{mill:07}
  J. van Mill, \emph{Homogeneous compacta}, Open problem in
  topology II, pp.189--195, Elsevier, 2007.
%
%
%

  \bibitem{zsilinszky:98}
  L. Zsilinszky, \emph{Polishness of the Wijsman topology
  revisited}, Proc. Amer. Math. Soc. {\bf 126} (1998),
  2575--2584.

%
%

  \end{thebibliography}
   \end{document}